\documentclass[11pt,a4paper]{article}
\usepackage{amssymb, latexsym, amsmath, amsthm}
\usepackage{amsfonts,amsmath,amsthm,amssymb}
\usepackage[english]{babel}
\usepackage{a4wide}
\usepackage{color}
\usepackage{amsfonts}

\newtheorem{thm}{Theorem}[section]
\newtheorem{lemma}[thm]{Lemma}
\newtheorem{prop}[thm]{Proposition}
\newtheorem{defn}[thm]{Definition}
\newtheorem{cor}[thm]{Corollary}
\newtheorem{remark}[thm]{Remark}
\newtheorem{remarks}[thm]{Remarks}
\newtheorem{ex}[thm]{Example}

\theoremstyle{definition}

 \newcommand{\T}{A[y_1;\sigma_1, d_1][y_2; \sigma_2, d_2]}
\newcommand{\A}[1]{{#1}_P[y_1,y_2;\sigma,\delta,\tau]}
\newcommand{\si}{\sigma}
\newcommand{\de}{\delta}

\newcommand{\Z}{\mathbb{Z}}

\usepackage{color}

\begin{document}
 \date{}

\title{Double Ore Extensions versus Iterated Ore Extensions}
\author{Paula A. A. B. Carvalho$^\dagger$, Samuel A.\ Lopes$^\dagger$ and Jerzy Matczuk$^\ast$ \\
\\
$^\dagger$ Departamento de Matem\'atica Pura\\ Faculdade de Ci\^encias da Universidade do Porto \\
R.Campo Alegre 687, 4169-007 Porto, Portugal\\
pbcarval@fc.up.pt; slopes@fc.up.pt \\
\\
$^\ast$ Institute of Mathematics, Warsaw University,\\
 Banacha 2, 02-097 Warsaw, Poland\\
  jmatczuk@mimuw.edu.pl          }

\maketitle

\begin{abstract}
\noindent Motivated by the construction of new examples of
Artin-Schelter regular algebras of global dimension four,  J.J.Zhang and J.Zhang (2008) 
introduced   an algebra extension $A_P[y_1,y_2;\sigma,\delta,\tau ]$ of $A$, which they  called a   double Ore extension.
This construction   seems to be similar to that of a  two-step
 iterated Ore extension over $A$.  The aim of this paper is to  describe those  double
Ore extensions which can be presented as iterated Ore extensions of
the form $A[y_1;\sigma_1, \delta_1][y_2;\sigma_2, \delta_2]$. We also give  partial answers  to some questions posed in Zhang and Zhang (2008).
\end{abstract}


\section*{Introduction}


In 2008, J.J. Zhang and J. Zhang introduced 
a new
construction for extending a given algebra $A$,  by simultaneously
adjoining two generators, $y_1$ and $y_2$.  This construction
resembles that of an Ore extension, and it was indeed called a
double Ore extension (or double extension, for short).
It should be noted that  there are no inclusions between  the classes of all double extensions of an algebra $A$ and of all length two iterated Ore extensions of $A$.
The aim of this paper is to describe the common part of these two classes of extensions of  $A$.

In Section~\ref{S:doe} we parallel the constructions of double
extensions and Ore extensions, taking the opportunity to correct
some typos which occurred in Zhang and Zhang (2008, p.~2674)
and again in Zhang and Zhang (2009, p.~379),
concerning the relations that the data of a double extension must
satisfy. In Section~\ref{S:ispr} we present our main results,
Theorems \ref{DOasO} and \ref{y2y1 version}, which offer necessary
and sufficient conditions for a double extension
$A_P[y_1,y_2;\sigma,\delta,\tau ]$ to be presented as iterated Ore
extensions of the form $A[y_1;\si_1,\de_1][y_2;\si_2,\de_2]$ or
$A[y_2;\si_2,\de_2][y_1;\si_1\de_1]$. These, along with  Lemma
\ref{diagonalizavel}, give necessary conditions for a double
extension $A_P[y_1,y_2;\sigma,\delta,\tau]$ to be presented as an
iterated Ore extension $A[x_1;\si_1,\de_1][x_2;\si_2,\de_2]$, with
$x_1$ and $x_2$ a basis of the vector space spanned by $y_1$ and
$y_2$.

In Zhang and Zhang (2009), the authors pursue the study of
Artin-Schelter regular algebras of global dimension four, by
classifying certain types of double extensions and establishing
some of their properties. So as to simplify their task, they
develop criteria for a double extension, of the type they
considered, to be an iterated Ore extension. This is obtained in
Zhang and Zhang (2008, Proposition 3.6), which is a special case
of Theorems \ref{DOasO} and \ref{y2y1 version} below.

We conclude with some applications and give  partial answers to
some questions posed in Zhang and Zhang (2008).


\section{Double Ore Extensions}\label{S:doe}

Throughout this paper, $K$ denotes a field of arbitrary
characteristic and $K^*$ is its multiplicative group of units.
For a $K$-algebra $B$, the  algebra  of $n$ by $m$ matrices with entries in $B$
will be denoted by $M_{n\times m}(B)$.

Let $A$ be a subalgebra of a $K$-algebra $R$ and $x\in R$ be such that $R$
is a free left $A$-module with basis $\{x^i\}_{i=0}^\infty$ and
$xA\subseteq Ax+A$. Then, for any $ a\in A$, there exist $\si(a),
 d(a)\in A$ such that $xa=\si(a)x+ d (a)$. It is well known
(cf. Cohn (1971)) that the above conditions imply that $\si$ is an
endomorphism of $A$ and $d$ is a $\si$-derivation of $A$, i.e.,
$d$ is a $K$-linear  map such that $d(ab)=\si(a) d(b)+d(a)b$, for
all $a,b\in A$. Conversely, if an endomorphism $\si$ and a
$\si$-derivation $d$ of a $K$-algebra $A$ are given, then the
multiplication in $A$ and the condition $xa=\si(a)x+d(a)$ induce a
structure of an associative $K$-algebra on the free left
$A$-module with  basis $\{x^i\}_{i=0}^\infty$. This extension is
called an Ore extension and is denoted by $A[x;\si,d]$. One can
easily check that the Ore extension $A[x;\si,d]$ is a free right
$A$-module with  basis $\{x^i\}_{i=0}^\infty$ if and only if $\si$
is an automorphism of $A$  if and only if $\si$ is injective and
$xA+A=Ax+A$.

We will now recall  the definition of a double  extension, as given in
Zhang and Zhang (2008).

\begin{defn}Let $A$ be a subalgebra of a $K$-algebra $B$. Then:
\begin{enumerate}
\item $B$ is called a right double extension of $A$ if:
\begin{itemize}
 \item[(i)] $B$ is generated by $A$ and two new variables $y_1$ and $y_2$;
 \item[(ii)] $y_1$ and $y_2$ satisfy the relation
 \begin{equation}
 y_2y_1=p_{12}y_1y_2+p_{11}y_1^2+\tau_1 y_1+\tau_2 y_2+\tau_0, \label{def:DOE1}
 \end{equation}
 for some $p_{12},p_{11}\in K$ and $\tau_1, \tau_2, \tau_0\in
 A$;
 \item[(iii)] $B$ is a free left $A$-module with
 basis $\{y_1^{i}y_2^{j}: i, j\geq 0\}$;
 \item[(iv)] $y_1A+y_2A+A\subseteq Ay_1+Ay_2+A$.
 \end{itemize}
 \item A right  double extension $B$ of $A$ is called a double
 extension if:
 \begin{itemize}
 \item[(i)] $p_{12}\ne 0$;
 \item[(ii)] $B$ is a free right $A$-module with basis
 $\{y_2^{i}y_1^{j}: i, j\geq 0\}$;
 \item[(iii)] $y_1A+y_2A+A= Ay_1+Ay_2+A$.
 \end{itemize}
\end{enumerate}
\end{defn}

\bigskip\noindent
Condition (a)$(iv)$ from the above definition is equivalent to the existence of two maps
 $$\sigma=\left[\begin{array}{cc} \sigma_{11}  &
 \sigma_{12}  \\ \sigma_{21}  & \sigma_{22}  \end{array}\right]\colon A\rightarrow M_{2\times
 2}(A)\quad
\mbox{and}\quad \de=\left[\begin{array}{c} \de_{1}\\  \de_{2}
\end{array}\right]\colon A\rightarrow M_{2\times 1}(A),$$
such that
 \begin{equation}
 \left[\begin{array}{c} y_1 \\ y_2 \end{array}\right]a=\sigma(a)
 \left[\begin{array}{c} y_1 \\ y_2 \end{array}\right]+\delta(a),\;
  \mbox{for all}\;\; a\in A. \label{def:DOE2}\end{equation}
 In case $B$ is a right double extension of $A$,
  we will write $B=A_P[y_1,y_2;\sigma,\delta,\tau]$,
  where $P=\{p_{12},p_{11}\}\subseteq K$,  $\tau=\{\tau_0,\tau_1,\tau_2\}\subseteq A$
  and $\si, \de$ are as above.
   The set $P$ is called a \emph{parameter} and $\tau$ a \emph{tail}.

Suppose $A_P[y_1,y_2;\sigma,\delta,\tau]$ is a right double extension. Then,
it is clear  that all maps  $\sigma_{ij}$ and $\de_i$
  are   endomorphisms of the $K$-vector space $A$. In Zhang and Zhang (2008, Lemma 1.7) the authors showed
 that  $\si$ must be a homomorphism of algebras and $\de$  a
 $\si$-derivation, in the sense that $\de$ is $K$-linear and
 satisfies $\de(ab)=\si(a)\de(b)+\de(a)b$, for all $a,b\in A$. One can easily
  check that, if the matrix $ \left[\begin{array}{cc} \sigma_{11}  &
 \sigma_{12}  \\ \sigma_{21}  & \sigma_{22}  \end{array}\right]$
   is triangular, then both $\si_{11}$ and
 $\si_{22}$ are algebra homomorphisms.

It is known that a map $d\colon A\rightarrow A $ is a
$\si$-derivation, where $\si$ is an endomorphism of $A$, if and only if the
map from $A$ to $M_{2\times 2}(A)$ sending $a$ onto
  $\left[\begin{array}{cc}
    \si (a) & d(a)\\
    0 & a
  \end{array}\right]$
  is a homomorphism of algebras. This, in particular, implies
  that for any algebra endomorphism $\si$ of $K[x]$  and any polynomial $w\in
  K[x]$, there exists a (unique) $\si$-derivation $d$ of $K[x]$ such that
  $d(x)=w$.

Let us observe that if $\tau\subseteq K$, then the subalgebra of
$\A{A}$ generated by $y_1$ and $y_2$ is the double extension $K_P[y_1,y_2;\sigma',\delta',\tau]$, where $\si'=\si \mid_{K}$ is the canonical embedding of
$K$ in $M_{2\times 2}(K)$ and $\de'=\de \mid_{K}=0$ is the zero map. The
following proposition shows that the latter is always an iterated
Ore extension.

\begin{prop}
\label{rem1}  Let $B=K_P[y_1,y_2;\sigma',\delta',\tau]$. Then $B\simeq K[x_1][x_2;\si_2,d_2]$ is an
iterated Ore extension, where $\si_2$ is the algebra endomorphism
of the polynomial ring $K[x_1]$ defined by
$\si_2(x_1)=p_{12}x_1+\tau_2$ and $d_2$ is the
$\si_2$-derivation  of $K[x_1]$ given by
$d_2(x_1)=p_{11}x_1^2+\tau_1x_1+\tau_0$. Moreover, $B$ is a double
extension of $K$ if and only if $p_{12}\ne 0$.
\end{prop}
\begin{proof} The preceding remarks
guarantee the existence (and uniqueness) of the $\si_2$-derivation $d_2$. Thus, the iterated Ore
extension  $K[x_1][x_2;\si_2,d_2]$ can be considered. It is routine to  check that
$x_2x_1=p_{12}x_1x_2+p_{11}x_1^2+\tau_1 x_1+\tau_2 x_2+\tau_0$ holds in
$K[x_1][x_2;\si_2,d_2]$. This means that there is an algebra
homomorphism from $K[x_1][x_2;\si_2,\de_2]$ onto $B$ mapping
$x_i$ to $y_i$, $i=1,2$. Since $\{x_1^i x_2^j\mid i,j\geq 0\}$ and $\{y_1^iy_2^j\mid i,j\geq 0\}$ are
 bases of $K[x_1][x_2;\si_2,d_2]$ and $B$ over $K$, respectively, the homomorphism is an isomorphism.

If $p_{12}\ne 0$, then $\si_2$ is an automorphism of $K[x_1]$.
This implies that the set $\{x_2^ix_1^j\mid i,j\geq 0\}$ is a
basis of $K[x_1][x_2;\si_2,d_2]$ as a (right) $K$-vector space, and thus the same is true for $\{y_2^i y_1^j\mid i,j\geq 0\}$
and $B$.
Hence, $B$ is a double extension.
\end{proof}

\begin{remark}
 Let $C=K[y_1][y_2;\si_2,d_2]$ be as in Proposition \ref{rem1}. Then, for any $K$-algebra $A$, we have:
$$A\otimes_KC=A_P[y_1,y_2;\sigma,\delta,\tau],$$
where $\si=\left[\begin{array}{cc}
    id_A & 0 \\
    0 & id_A
  \end{array}\right]
$,  $\de=
  \left[\begin{array}{c}
    0   \\
    0
  \end{array}\right]$, $P=\{p_{12}, p_{11}\}$ and $\tau=\{\tau_0,\tau_1,\tau_2\}$ are
as in Proposition~\ref{rem1}.
\end{remark}

\begin{prop}\label{existance}
 Given  $P=\{p_{12},p_{11}\}\subseteq K$,  $\tau=\{\tau_0,\tau_1,\tau_2\}\subseteq
 K$,   $\sigma\colon A\rightarrow M_{2\times 2}(A)$ an algebra
 homomorphism
and $\de\colon A\rightarrow M_{2\times 1}(A)$ a $\si$-derivation,
let $C=K[y_1][y_2;\si_2,d_2]$ be as in Proposition~\ref{rem1}. Then, the following conditions are equivalent:
\begin{enumerate}
 \item the right double extension $\A{A}$ exists;
 \item one can extend the multiplications from $A$ and $C$ to
a multiplication in the vector space $A\otimes_KC$,
 satisfying $\left[\begin{array}{c} y_1 \\ y_2
\end{array}\right]a=\sigma(a)
 \left[\begin{array}{c} y_1 \\ y_2
  \end{array}\right]+\delta(a)$, for all $a\in A$.
\end{enumerate}
\end{prop}
\begin{proof}
 The remark just before Proposition \ref{rem1} implies that for any sets
 $P$ and $\tau$ of data, with $\tau\subseteq K$, the iterated Ore extension $C$ exists.
 Now it is easy to complete the proof by using Proposition~\ref{rem1} and the
 definition of a right double extension.
\end{proof}

Using Bergman's diamond lemma (Bergman (1978)), Zhang and Zhang gave
a universal construction for a right double extension.
Unfortunately, there are three small typos in  the relations
(R3.4)--(R3.6) appearing in Zhang and Zhang (2008, p.\ 2674).
These come originally from analogous typos in Zhang and Zhang
(2008, Lemma 1.10(a) and Equation
(E1.10.3)). For the convenience of the reader, we re-write
relations (R3.1)--(R3.6) of Zhang and Zhang (2008) as relations
(\ref{Rel:R31})--(\ref{Rel:R36}) below, with the corrected typos
underlined.

 \begin{prop}\label{Zhang}
 (Zhang and Zhand, 2008, Lemma 1.10, Proposition 1.11)
 Given a  $K$-algebra $A$, let $\sigma$ be a homomorphism from $A$ to
 $M_{2\times 2}(A)$, $\delta$  a $\sigma$-derivation
 from $A$ to $M_{2\times 1}(A)$, $P=\{p_{12},p_{11}\}$  a  set  of
 elements of $K$ and $\tau=\{\tau_0,\tau_1,\tau_2\}$ a set
  of elements of $A$.
 Then, the associative   $K$-algebra $B$  generated by $A$, $y_1$ and $y_2$, subject
 to the relations
(\ref{def:DOE1}) and (\ref{def:DOE2}), is a right double extension if and
only if the maps $\si_{ij}$ and $\rho_{k}$, $i\in\{1,2\}$, $j, k\in\{ 0, 1, 2\}$,
satisfy the six relations (\ref{Rel:R31})--(\ref{Rel:R36}) below, where $\si_{i0}=\de_i$ and $\rho_{k}$ is a \underline{right} multiplication by $\tau_{k}$.
\end{prop}

\begin{align}
\label{Rel:R31}
\sigma_{21} \sigma_{11}+p_{11}\sigma_{22}  \sigma_{11}=&\  p_{11}\sigma_{11}^{2}+p_{11}^2\sigma_{12}  \sigma_{11}
+p_{12}\sigma_{11}  \sigma_{21} +p_{11}p_{12}\sigma_{12}  \sigma_{21}\\
\label{Rel:R32}
\sigma_{21}\sigma_{12}+p_{12}\sigma_{22}\sigma_{11}=&\  p_{11}\sigma_{11}\sigma_{12}+p_{11}p_{12}\sigma_{12}\sigma_{11}
+p_{12}\sigma_{11}\sigma_{22}+p_{12}^2 \sigma_{12}\sigma_{21}\\
\label{Rel:R33}
\sigma_{22}\sigma_{12}=&\  p_{11}\sigma_{12}^2 +p_{12}\sigma_{12}\sigma_{22}\\
\label{Rel:R34}
\sigma_{20}\sigma_{11} + \sigma_{21}\sigma_{10} +\underline{\rho_{1}\sigma_{22}\sigma_{11}} =&\  p_{11} \left( \sigma_{10}\sigma_{11}+\sigma_{11}\sigma_{10}+\tau_{1}\sigma_{12}\sigma_{11}  \right)\nonumber \\
& +p_{12} \left( \sigma_{10}\sigma_{21}+\sigma_{11}\sigma_{20}+\tau_{1}\sigma_{12}\sigma_{21}  \right)
+\tau_{1}\sigma_{11}  +\tau_{2}\sigma_{21}\\
\label{Rel:R35}
\sigma_{20}\sigma_{12} + \sigma_{22}\sigma_{10} +\underline{\rho_{2}\sigma_{22}\sigma_{11}} =&\  p_{11} \left( \sigma_{10}\sigma_{12}+\sigma_{12}\sigma_{10}+\tau_{2}\sigma_{12}\sigma_{11}  \right)\nonumber \\
& +p_{12} \left( \sigma_{10}\sigma_{22}+\sigma_{12}\sigma_{20}+\tau_{2}\sigma_{12}\sigma_{21}  \right)
+\tau_{1}\sigma_{12}  +\tau_{2}\sigma_{22}\\
\label{Rel:R36}
\sigma_{20}\sigma_{10} + \underline{\rho_{0}\sigma_{22}\sigma_{11}} =&\  p_{11} \left( \sigma_{10}^2
+\tau_{0}\sigma_{12}\sigma_{11}  \right)\nonumber \\ & +p_{12} \left( \sigma_{10}\sigma_{20}+\tau_{0}\sigma_{12}\sigma_{21}  \right)
+\tau_{1}\sigma_{10}  +\tau_{2}\sigma_{20}+\tau_{0}id_A
\end{align}

\begin{remarks}\hfill
\begin{enumerate}
\item[\textup{1.}] Proposition~\ref{existance} can be used to obtain a direct proof of  Proposition \ref{Zhang}, i.e., one which does not use Bergman's diamond lemma, provided that
$\tau=\{\tau_0,\tau_1,\tau_2\}\subseteq K$.
\item[\textup{2.}] Proposition~\ref{Zhang} implies the uniqueness, up to isomorphism, of a right double extension of $A$, with given   $\sigma$, $\delta$, $P$ and $\tau$, provided such an extension exists. Indeed, assume $\overline{B}=\A{A}$ is a right double extension of $A$.
Then, by Zhang and Zhang (2008, Lemmas 1.7 and 1.10(b)), the data
$\sigma$, $\delta$, $P$ and $\tau$ satisfy the conditions of
Proposition~\ref{Zhang}. Let $B$ be as in this proposition. Then,
there is an algebra homomorphism from $B$ to $\overline{B}$ which
restricts to the identity on $A$ and maps $y_{i}\in B$ to the
corresponding element $y_{i}\in \overline{B}$, $i=1, 2$. Since $B$
is a free left $A$-module with basis $\{y_1^{i}y_2^{j}: i, j\geq
0\}$ and the same holds for $\overline{B}$, this map is an
isomorphism, thus proving uniqueness.
\end{enumerate}
\end{remarks}

 As noticed in Zhang and Zhang (2008, Remark 1.4), by choosing a suitable basis of the vector
 space $Ky_1+Ky_2$,  we can prove:

 \begin{lemma}\label{L:newP}
 Let $B=A_P[y_1,y_2;\sigma,\delta,\tau]$ be a right double extension.
 \begin{enumerate}
 \item If $p_{11}\neq 0$ and $p_{12}=1$, then
 $$B\simeq A_{\{1,1\}}\left[\overline{y}_1,\overline{y}_2;\left[\begin{array}{cc} \sigma_{11} &
 p_{11}\sigma_{12}\\ p_{11}^{-1}\sigma_{21} & \sigma_{22}\end{array}\right],
 \left[\begin{array}{c} p_{11}\delta_{1}\\ \delta_{2}\end{array}\right], \overline{\tau}\right]$$
 where $\overline{\tau}=\{p_{11}\tau_0,\tau_1,p_{11}\tau_2\}$, $\overline{y}_1=p_{11}y_1$ and $\overline{y}_2=y_2$.
\item If  $p_{12}\neq 1$, then
 $$B\simeq  A_{\{p_{12},0\}}\left[\overline{y}_1,\overline{y}_2;\left[\begin{array}{cc} \sigma_{11}-q\sigma_{12} & \sigma_{12}\\ \sigma_{21}+q(\sigma_{11}-\sigma_{22})-q^2\sigma_{12} & \sigma_{22}+q\sigma_{12}\end{array}\right],
 \left[\begin{array}{c} \delta_{1}\\ \delta_{2}+q\delta_1\end{array}\right], \overline{\tau} \right]$$
 where $q=\frac{p_{11}}{p_{12}-1}$,  $\overline{\tau}=\{\tau_0,\tau_1-q\tau_2,\tau_2\}$, $\overline{y}_1=y_1$ and $\overline{y}_2=y_2+qy_1$.

 \end{enumerate}
 \end{lemma}

Let $B=\A{A}$ be a right double extension and suppose
that $p_{12}\ne 1$. Then, as  observed above, by choosing adequate generators $\overline{y}_{i}$
and (possibly) modifying the data $\sigma$, $\delta$, $\tau$, one can assume that $p_{11}=0$.
Now suppose  $\overline{B}={A}_P[\overline{y}_1,\overline{y}_2;\overline{\sigma}, \overline{\delta}, \overline{\tau}]$ is a right
double extension with
$p_{11}=0$. Then  $\overline{B}$ has a natural
filtration, given by setting $\deg A=0$ and $\deg \overline{y}_1=\deg \overline{y}_2=1$. One can
check, in view  of   relations (\ref{def:DOE1}) and (\ref{def:DOE2}), that  the
associated graded algebra $\mathfrak{gr}(B)$ is isomorphic to
$A_P[\overline{y}_1,\overline{y}_2; \overline{\sigma}, 0, \{0,0,0\}]$. The above
shows that the following holds:

\begin{cor}\label{C:grB}
 Suppose that $B=\A{A}$ is a right double extension of $A$, with $p_{12}\ne 1$. Then, there exists a filtration on $B$
 such that the associated graded algebra $D=\mathfrak{ gr}(B)$ can be presented as follows:
 $D$ is generated over $A$ by indeterminates $z_1, z_2$;   it is   free as a left $A$-module with
 basis $\{z_1^{i}z_2^{j}: i, j\geq 0\}$; multiplication in $D$ is given by multiplication in $A$ and the conditions
 $z_2 z_1=p_{12}z_1 z_2$  and $z_1 A+z_2 A\subseteq Az_1+Az_2$, with
 $
 \left[\begin{array}{c} z_1 \\ z_2 \end{array}\right]a=\overline{\sigma}(a)
 \left[\begin{array}{c} z_1 \\ z_2 \end{array}\right]$, where $\overline{\sigma}$ is obtained from $\sigma$ and $q=\frac{p_{11}}{p_{12}-1}$ as in
 Lemma~\ref{L:newP}(b).

 Furthermore, in case $B$ is a double extension, then  $D$ is also free as a right $A$-module with
 basis $\{z_2^{i}z_1^{j}: i, j\geq 0\}$ and $z_1 A+z_2 A = Az_1+Az_2$.
\end{cor}

 Suppose that $\mathcal{P}$ is a ring-theoretical property which passes from the associated graded algebra $\mathfrak{gr}(C)$
 to the (filtered) algebra $C$.  The above yields that, while investigating the lifting of property $\mathcal{P}$
 from  $A$ to a right double  extension $B$ of $A$, one needs only consider two cases:
 $P=\{1,1\}$ and  $P=\{p_{12},0\}$, with $p_{12}\in K$.

\section{Double extensions as iterated skew polynomial rings}\label{S:ispr}

In general, an iterated Ore extension of the form $\T$ is not a
right double extension.  In spite of this, one can check that if
$$\begin{array}{c}
   \sigma_2(A)\subseteq A, \quad \sigma_2(y_1)=p_{12}y_1+\tau_2, \\
   d_2(A)\subseteq Ay_1+A,\quad
   d_2(y_1)=p_{11}y_1^2+\tau_1y_1+\tau_0
   \end{array}$$
where $p_{ij}\in K$ and $\tau_i\in A$, then the given iterated Ore
extension $\T$ is indeed a right double extension $\A{A}$, with
$P=\{p_{12},p_{11}\}$, $\tau=\{\tau_0,\tau_1,\tau_2\}$,
$ \sigma=
  \left[\begin{array}{cc}
   \si_1   &  0\\
     \si_{21} &  \si _2|_A
  \end{array}\right]
$ and
$\de=
  \left[\begin{array}{c}
    d_1  \\
    \de_2
  \end{array}\right]
$, where $\si_{21}, \de_2  \colon A\rightarrow A$ are defined
by the condition $d_2(a)=\si_{21}(a)y_1+\de_2(a)\in Ay_1+A$, for $a\in A$.

In our next theorem, we give necessary and sufficient conditions for $\A{A}$
to be an iterated Ore extension of the form $\T$. By this we mean that we determine when the identity map on $A$
extends to an algebra isomorphism from $\A{A}$ to $\T$ sending $y_{i}$ to $y_{i}$, $i=1, 2$. To proceed with this, we need the following:

\begin{lemma}\label{commuting}
Let   $\T$ be an iterated Ore extension such that $\si_2(A)\subseteq A$ and $\si_2(y_1)= py_1+q$, for some $p\in K^*$ and $q\in A$.
Then, $\si_1\si_2(a)=\si_2\si_1(a)$, for all $a\in A$.
\end{lemma}
\begin{proof}
Let $a\in A$. Applying $\si_2$ to the equality
$y_1a=\si_1(a)y_1+d_1(a)$, we obtain
$p\si_1\si_2(a)y_1+pd_1\si_2(a)+q\si_2(a)=p\si_2\si_1(a)y_1+\si_2\si_1(a)q +
\si_2d_1(a)$. Since $p\in K^*$, the thesis follows.
\end{proof}


%

\begin{thm}\label{DOasO}
Let $A,B$ be $K$-algebras such that $B$ is an extension of $A$. Assume
$P=\{p_{12},p_{11}\}\subseteq K$,
$\tau=\{\tau_0,\tau_1,\tau_2\}\subseteq A$, $\sigma$ is an algebra
homomorphism from $A$ to $M_{2\times 2}(A)$ and $\delta$ is a
$\sigma$-derivation from $A$ to $M_{2\times 1}(A)$.
\begin{enumerate}
\item The following conditions are equivalent:

\begin{enumerate}
 \item $B=A_P[y_1,y_2;\sigma,\delta,\tau]$ is a right double
 extension of $A$ which can be presented as an iterated
 Ore extension $\T$;
 \item  $B=A_P[y_1,y_2;\sigma,\delta,\tau]$ is
 a right double extension of $A$ with $\sigma_{12}=0$;
 \item  $B=\T$ is an iterated Ore extension such
 that
 $$\begin{array}{c}
   \sigma_2(A)\subseteq A, \quad \sigma_2(y_1)=p_{12}y_1+\tau_2,
   \\
   d_2(A)\subseteq Ay_1+A,\quad
   d_2(y_1)=p_{11}y_1^2+\tau_1y_1+\tau_0,
   \end{array}$$
for some $p_{ij}\in K$ and $\tau_i\in A$.
The maps $\sigma$, $\delta$, $\si_i$ and $\de_i$, $i=1,2$, are related by:
$ \sigma= \left[\begin{array}{cc}
   \si_1   &  0\\
   \si_{21} &  \si _2|_A
  \end{array}\right]
$, $\de(a)= \left[\begin{array}{c}
    d_1(a)  \\
d_2(a)-\sigma_{21}(a)y_1
  \end{array}\right]
$, for all $a\in A$.
\end{enumerate}

\item If any of the statements from (a)   holds, then  $B$ is a double
extension of $A$ if and only if $\sigma_1=\sigma_{11}$ and $\sigma_2|_A=\sigma_{22}$
are
automorphisms of $A$ and $p_{12}\neq 0$.
\end{enumerate}
\end{thm}

\begin{proof}
Notice that, if  $B=A_P[y_1,y_2;\sigma,\delta,\tau]$, then $y_1A\subseteq Ay_1+A$ if and only if
$\sigma_{12}=0$. Therefore $(a)(i)\Rightarrow (a)(ii)$.

Suppose now that $(a)(ii)$ holds.  Then  $y_1A\subseteq Ay_1+A$. Hence,
every element of the subalgebra $A[y_1]$ of $B$ can be written in the
form $a_ny_1^n+\ldots + a_0$, for suitable $n\geq 0$ and $a_i\in A$. Since
$B$ is a free left $A$-module with basis
$\{y_1^i y_2^j\}_{i, j=0}^\infty$, the elements
$\{y_1^i \}_{i=0}^\infty$ are $A$-independent, i.e., $A[y_1]$ is a
free left $A$-module with that basis. Multiplication in $A[y_1]$ is given by
multiplication in $A$ and the condition $y_1a=\si_1(a)y_1 +
d_1(a)$, where $\si_1=\si_{11}$ and $d_1=\de_1$. Thus,
$A[y_1]=A[y_1;\si_1,d_1]$.

Since $B$ is a right double extension of $A$, $B$ is a free left
$A[y_1;\sigma_1, d_1]$-module with basis $\{y_2^i \}_{i=0}^\infty$.
Relation (\ref{def:DOE1}) can be re-written as
\begin{equation}
y_2y_1=(p_{12}y_1+\tau_2)y_2+p_{11}y_1^2+\tau_1 y_1+\tau_0
\label{DOE3}
\end{equation}
and, by (\ref{def:DOE2}), we also have:
\begin{equation}
y_2A\subseteq Ay_2+(Ay_1+A).\label{DOE4}
\end{equation}
Thus, by the above,
\begin{equation}
y_2A[y_1]\subseteq A[y_1]y_2+A[y_1]\label{DOE5}.
\end{equation}
This means that $B=\T$, for some endomorphism $\sigma_{2}$ and some
$\sigma_{2}$-derivation $d_{2}$ of $A[y_1;\si_1,d_1]$.
Conditions (\ref{DOE3}) and (\ref{DOE4})
imply that $\sigma_2(y_1)=p_{12}y_1+\tau_2$,  $
\sigma_2(A)\subseteq A$, $ d_2(A)\subseteq Ay_1+A$ and
$d_2(y_1)=p_{11}y_1^2+\tau_1y_1+\tau_0$. By (\ref{def:DOE2}),
$$y_2a=\sigma_{22}(a)y_2+\sigma_{21}(a)y_1+\delta_2(a),\;\; \mbox{for all}\;\; a\in
A.$$ Thus, we also have   $\sigma_2(a)=\sigma_{22}(a)$ and
$d_2(a)=\sigma_{21}(a)y_1+\delta_2(a)$, for all $a\in A$. Hence
$(a)(iii)$ holds.

As observed at the beginning of this section,
$(a)(iii)\Rightarrow(a)(i)$ holds, and the proof of $(a)$ is
completed.

Assume now that $B$ is a double extension, with $\sigma_{12}=0$.
Then, by definition, $p_{12}\neq 0$. J.J.~Zhang and J.~Zhang
introduced the determinant of $\si$,  $\det(\si)\colon
A\rightarrow A$,  by setting
$\det(\si)=-p_{11}\si_{12}\si_{11}+\si_{22}\si_{11}-p_{12}\si_{12}\si_{21}$,
and showed  (cf. Zhang and Zhang (2008,~Lemma 1.9 and Proposition
2.1(a)(b))) that $\det(\si)$ is an automorphism of $A$, provided
that $B$ is a double   extension of $A$. As $\si_{12}=0$, this
implies that $\det(\si)=\si_{22}\si_{11}$ is invertible in
End$_K(A)$. Notice that Lemma \ref{commuting} and (a)(iii) above
yield that $\sigma_{11}$ and $\sigma_{22}$ commute. Therefore,
both $\sigma_{11}$ and $\sigma_{22}$ are automorphisms of $A$.

Conversely, suppose that $\sigma_1=\sigma_{11}$, $\sigma_2|_A=\sigma_{22}$
are automorphism of $A$ and  $p_{12}\neq 0$. Since $\sigma_2(y_1)=p_{12}y_1+\tau_2$, this implies that
$\si_2$ is an automorphism of $A[y_1;\si_1,d_1]$. Hence,
$\{y_2^{i}y_1^{j}: i, j\geq 0\}$ is a basis of $B=\T$
as a right $A$-module, i.e., $B$ is a double extension of $A$. This
completes the proof of $(b)$.
\end{proof}

The following  lemma gives a necessary and sufficient condition
for the matrix corresponding to $\sigma$  to be triangularizable,
by choosing adequate generators of $A_P[y_1,y_2;\sigma,\delta,\tau]$ from $Ky_1+Ky_2$, i.e., it gives a necessary
condition for a right double extension $\A{A}$ to be presented as
an iterated Ore extension over $A$.

\begin{lemma}\label{diagonalizavel}
Let $B=A_P[y_1,y_2;\sigma,\delta,\tau]$ be a right double
extension, $k,l\in K$ and $0\ne z=ky_1+ly_2\in B$. Then:
\begin{center}
$zA\subseteq Az+A\qquad\mbox{iff}\qquad
kl\sigma_{11}+l^2\sigma_{21}=kl\sigma_{22}+k^2\sigma_{12}.$
\end{center}
\end{lemma}
\begin{proof} If either $k=0$ or $l=0$, then the identity above
reduces to  $\si_{21}=0$ or $\si_{12}=0$, accordingly. This
gives the thesis in this case.

Suppose $k,l\in K^*$.  Let $a\in A$. One can compute that
\begin{center}
$za=\left(\sigma_{11}(a)+\frac{l}{k}\sigma_{21}(a)\right)ky_1+\left(\frac{k}{l}\sigma_{12}(a)+\sigma_{22}(a)\right)ly_2+k\delta_1(a)+l\delta_2(a).$
\end{center}
This yields the thesis.
\end{proof}

Suppose that  the right double extension $B=\A{A}$ can be  presented as an iterated Ore
extension of the form $A[y_2,\si_2',d_2'][y_1;\si_1',d_1']$.  Then, we must have
$\sigma_{21}=0$, as $ y_2A\subseteq Ay_2+A$.  Notice also that
$p_{12}$ has to be nonzero, as otherwise the quadratic relation
 (\ref{def:DOE1})  would imply that $\{ y_{1}^{i}\}_{i=0}^\infty$ is not a free basis of $B$ as a left
$A[y_2;\si_2',d_2']$-module. Now, relation (\ref{def:DOE1}) together with
$y_1y_2\in A[y_2,\si_2',d_2']y_1+A[y_2,\si_2',d_2']$ imply that
$p_{11}=0$.   In this case, the quadratic relation
(\ref{def:DOE1}) becomes
\begin{equation}
y_2y_1=p_{12}y_1y_2+\tau_1y_1+\tau_2y_2+\tau_0. \label{sym rel}
\end{equation}
Observe that, in any right double extension $B$ satisfying
relation (\ref{sym rel}), the set $\{y_2^{i}y_1^{j}: i, j\geq 0\}$ still
forms a basis of $B$ as a left $A$-module. The
fact that this relation  is left-right symmetric
implies that there is an isomorphism
$$B\simeq A_{\{p_{12}^{-1},0\}}\left[y_2,y_1;\left[\begin{array}{cc} \sigma_{22} &
 \sigma_{21}\\ \sigma_{12} & \sigma_{11}\end{array}\right],
 \left[\begin{array}{c} \delta_{2}\\ \delta_{1}\end{array}\right],
 \{-p_{12}^{-1}\tau_0,-p_{12}^{-1}\tau_1,-p_{12}^{-1}\tau_2\}\right]$$
realized by interchanging the roles of $y_1$ and $y_2$.

The  remarks above, together with Theorem \ref{DOasO}, yield the
following (cf. Zhang and Zhang (2009, Proposition 3.6)):

\begin{thm}\label{y2y1 version}
Let $B=A_P[y_1,y_2;\sigma,\delta,\tau]$ be a right double extension of the $K$-algebra $A$, where
$P=\{p_{12},p_{11}\}\subseteq K$,
$\tau=\{\tau_0,\tau_1,\tau_2\}\subseteq A$, $\sigma \colon
A\rightarrow M_{2\times 2}(A)$ is an algebra
homomorphism  and
$\delta\colon A \rightarrow M_{2\times 1}(A)$ is a
$\sigma$-derivation.
Then, $B$ can be presented as an iterated Ore extension
$A[y_2;\si_2',d_2'][y_1;\si_1',d_1']$ if and only if $\sigma_{21}=0$,
 $p_{12}\ne 0$ and $p_{11}=0$.
 In this case, $B$ is a double
extension if and only if $\sigma_2'=\sigma_{22}$ and
$\sigma_1'|_A=\sigma_{11}$ are automorphisms of $A$.
\end{thm}

Let $\mathcal{P}$ denote one of the following ring-theoretical
properties: being left (right) noetherian, being a domain, being
prime, being semiprime left (right) noetherian, being semiprime
left (right) Goldie. It is   known that $\mathcal{P}$ lifts from a
ring $R$ to an Ore extension $R[x; \si, d]$, provided that $\si$
is an automorphism of $R$ (cf. Lam (1997), Matczuk (1995),
McConnel and Robson (2001)). Thus, Theorems \ref{DOasO} and
\ref{y2y1 version} yield the following partial positive answers to
some of the questions posed in Zhang and Zhang (2008):

\begin{cor}\label{cornoetherian} Suppose that the $K$-algebra $A$
possesses the property $\mathcal{P}$. Then, the double extension
  $B=A_P[y_1,y_2;\sigma,\delta,\tau]$ also has the property
  $\mathcal{P}$, provided
that either $\sigma_{12}=0$, or $\sigma_{21}=0$ and
$p_{11}=0$.
\end{cor}

In Zhang and Zhang (2008), the authors asked whether primeness
(resp.\ semiprimeness) lifts from an algebra  $A$ to its double
extension $\A{A}$.  It is known that, in general, semiprimeness
does not lift from $A$ to an Ore extension
$A[y_1;\si_1,0]=A[y_1;\si_1]$, even if $\sigma_1$ is an
automorphism.   For such a non-semiprime extension  we know, by
Theorem \ref{DOasO}, that  $A[y_1;\sigma_1][y_2]$ is a double
extension, which is clearly not semiprime.  For a specific
example, one can take $A=\prod_{i\in \Z} K_i$, where $K_i=K$ is a
copy of the base field and $\sigma_1$ is the ``right shifting''
automorphism of $A$.  Then $ay_1A[y_1;\sigma_1]ay_1=0$, for
$a=(a_{i})$ with $a_0=1$ and $a_i=0$ if $i\neq 0$, i.e.,
$A[y_1;\sigma_1]$ is not semiprime. On the other hand,
semiprimeness does lift from the algebra $A$ to an Ore extension
$A[y_1;\si_1]$, provided that $\si_{1}$ is an automorphism and $A$
is noetherian. The problem of determining whether semiprimeness
lifts from A to $\A{A}$ when $A$ is noetherian still remains open.

One of the examples of a double extension which  appeared in Zhang
and Zhang (2008) is the following:

\begin{ex} Let $A=K[x]$ and fix
$a,b,c\in K$ with $b\neq 0$. Let $\sigma \colon
A\longrightarrow M_{2\times 2}(A)$ be the  algebra homomorphism given by
$\sigma(x)=\left[\begin{array}{cc} 0 & b^{-1}x\\ bx &
0\end{array}\right]$ and $\de\colon A\rightarrow M_{2\times
1}(A)$ be  the $\si$-derivation determined by the condition
  $\delta(x)=\left[\begin{array}{c} cx^2\\
-bcx^2\end{array}\right]$. Then, the double extension
$B^2(a,b,c)=A_{\{-1,0\}}[y_1,y_2;\sigma,\delta,\{0,0,ax^2\}]$
exists and it is the $K$-algebra generated by $x,y_1, y_2$, subject
to the relations:
\begin{equation} \label{aaaa}
y_2y_1=-y_1y_2+ax^2,\;\;\;y_1x=b^{-1}xy_2+cx^2,\;\;\;y_2x=bxy_1-bcx^2.
\end{equation}
\end{ex}

It was stated in Zhang and Zhang (2008) that if $a\ne 0$, then the
algebra $B^2(a,b,c)$ cannot be presented as an iterated Ore
extension over $K[x]$. The following proposition shows that this
is not so, in case the characteristic of the base field $K$ is
$2$.

\begin{prop} Let $a,b,c\in K$ with $b\neq 0$. The algebra $B^2=B^2(a,b,c)$ has the following
properties:
\begin{enumerate}
\item Suppose that $\mathrm{char}(K)=2$. Then $B^2$ is the differential operator algebra
$B^2=K[x,z][y_2; d]$, where $d$ is the derivation of $K[x,z]$
determined by $d(x)=xz-bcx^2$ and $d(z)=abx^2$. In particular, $B^2$ can be
presented as an iterated Ore extension over $A=K[x]$.
\item $B^2$ is a noetherian domain.
\end{enumerate}
\end{prop}
\begin{proof}
 Suppose that $\mathrm{char}(K)=2$. Let us take $z=by_1+y_2$. One can check, using (\ref{aaaa}), that $zx=xz$. In particular, $zA=Az$.
  Notice also that
$zA+y_2A+A=Az+Ay_2+A$, as $y_1A+y_2A+A=Ay_1+Ay_2+A$. Taking $y_2$
as the second generator of $B^2$ in the double extension, we have:
\begin{equation}\label{E:b2}
y_2z=-zy_2+2y_2^2+abx^2=zy_2+abx^2,
\end{equation}
given that $\mathrm{char}(K)=2$.
The above identity implies that $B^2$ can be also presented as a double
extension $B^2=A_{\{-1,0\}}[z,y_2;\sigma',\delta',\{0,0,abx^2\}]$,
for suitable maps $\si'$ and $\de'$. The condition $zA\subseteq
Az+A$ means that $\si'_{12}=0$. Thus, Theorem \ref{DOasO} implies
that $B^2$ is an iterated Ore extension over $A=K[x]$. In fact, using (\ref{aaaa}) and the characteristic of $K$,
one can check  that
\begin{equation}\label{a4}
y_2x=-xy_2+xz-bcx^2=xy_2+xz-bcx^2.
\end{equation}
Then, the identities $xz=zx$, (\ref{E:b2}) and (\ref{a4}) imply
that $B^2=K[x,z][y_2; d]$, where  $d$ is a derivation of $K[x,z]$ as
described above, i.e., $d(x)=xz-bcx^2$ and $d(z)=abx^2$. This proves (a).

For(b), let $\mathrm{char}(K)$ be arbitrary. Notice that, as $p_{11}=0$,  the algebra $B^2$ is filtered, as described in the paragraph preceding
Corollary~\ref{C:grB}. The associated
graded algebra, $\mathfrak{gr}(B^2)$, is generated by $x$, $z_{1}$ and $z_{2}$, subject to the
relations (cf.~(\ref{aaaa})):
\begin{align*}
z_2 z_1 &=-z_1 z_2;\\
z_1 x &=b^{-1}x z_2;\\
z_2 x &=bxz_1.
\end{align*}
Thus, $\mathfrak{gr}(B^2)$ is the iterated Ore extension
$$K[z_1][z_2; \sigma_1][x;\sigma_2],$$
where $\sigma_1(z_1)=-z_1$, $\sigma_2(z_2)=bz_1$,
$\sigma_2(z_1)=b^{-1}z_2$. Therefore, $\mathfrak{gr}(B^2)$ is a
noetherian domain, which implies that so is $B^2$.
\end{proof}

\section*{Acknowledgments}

 The two first
named authors were   supported  by \emph{Funda\c{c}\~ao para a
Ci\^encia e Tecnologia} (FCT),  through the \emph{Centro de
Matem\'atica da Universidade do Porto} (CMUP). The third named
author was   supported by Polish MNiSW grant No. N N201 268435.
The authors are also grateful for financial support from
GRICES/Ministry of Science and Higher Education of Poland, under
the project ``Rings with additional structures''.



\end{document}